\documentclass[11pt]{amsart} 

\usepackage[a4paper, margin=2.5cm]{geometry}

\usepackage[T1]{fontenc}    
\usepackage{lmodern}        

\usepackage{amsmath, amssymb, amsthm, amsrefs}
\usepackage{mathtools, mathrsfs, bm}
\usepackage{graphicx}
\usepackage{xcolor}

\definecolor{darkblue}{rgb}{0.0, 0.0, 0.55}
\usepackage{hyperref}
\hypersetup{
      colorlinks=true,
      linkcolor=darkblue,
      citecolor=darkblue,
      urlcolor=darkblue,
      linktoc=page,
}

\allowdisplaybreaks[4]
\numberwithin{equation}{section}

\def\le{\leqslant}
\def\ge{\geqslant}
\def \ds{\displaystyle}
\renewcommand{\Re}{\operatorname{Re}}
\renewcommand{\Im}{\operatorname{Im}}

\def \R{\mathbb{R}}
\def \C{\mathbb{C}}

\def \N{\mathbb{N}}

\DeclarePairedDelimiter{\norm}{\lVert}{\rVert}
\DeclarePairedDelimiter{\ev}{\langle}{\rangle}

\DeclarePairedDelimiterX\set[1]{\{}{\}}{%
      #1
}

\theoremstyle{plain}
\newtheorem{theorem}{Theorem}[section]

\newtheorem{lemma}[theorem]{Lemma}

\newtheorem{proposition}[theorem]{Proposition}
\theoremstyle{remark}
\newtheorem{remark}[theorem]{Remark}

\begin{document}
\title[Large-data $L^2$-decay for attractive-dissipative NLS]{Large-data $L^2$-decay for attractive-dissipative nonlinear Schr\"odinger equations without the strong dissipative condition}
\author[N. Kita]{Naoyasu Kita}
\address[]{Faculty of Advanced Science and Technology, Kumamoto University, Kumamoto, 860-8555, Japan}
\email{nkita@kumamoto-u.ac.jp}
\author[H. Miyazaki]{Hayato Miyazaki}
\address[]{Teacher Training Courses, Faculty of Education, Kagawa University, Takamatsu, Kagawa 760-8522, Japan}
\email{miyazaki.hayato@kagawa-u.ac.jp}
\author[T. Sato]{Takuya Sato}
\address[]{
Graduate School of Science and Engineering, Ehime University,
2-5 Bunkyo-cho, Matsuyama, Ehime, 790-8577, Japan
}
\email{
sato.takuya.gd@ehime-u.ac.jp
}
\date{}

\begin{abstract}
We prove a large-data $L^2$-decay estimate for nonlinear dissipative Schr\"odinger equations with attractive-dissipative power nonlinearity.
The main difficulty is the lack of sign definiteness of the standard energy when $\Re\lambda<0$, which prevents the usual energy argument from directly yielding a uniform gradient bound.

We introduce an augmented energy, obtained by adding a suitable multiple of the decreasing $L^2$-norm to the standard energy.
This produces an additional dissipative term and gives a direct uniform-in-time $H^1$ bound without the iteration argument used in previous works.
Consequently, for arbitrary initial data in the weighted energy space
$\Sigma = H^1\cap\mathcal{F}H^1$,
we obtain the decay rate previously known under the strong dissipative condition throughout the sharp decay range $1<p\le 1+2/d$.
This removes the remaining restriction $p\le 1+4/(3d)$ in the attractive-dissipative case.
\end{abstract}

\maketitle


\section{Introduction} 

We consider the nonlinear Schr\"odinger equation
\begin{align}
\begin{cases}
      \ds i \partial_t u  +  \frac{1}{2} \Delta u  = \lambda |u|^{p-1} u,
      \quad (t,x) \in (0,\infty)\times\R^d,\\
      u(0,x)=u_0(x), \quad x \in \R^d,
\end{cases}
      \label{nls} \tag{NLS}
\end{align}
where $d\in\N$, $p>1$, and $\lambda\in\C$ satisfies
\[
      \Im\lambda<0.
\]
Then the $L^2$-norm is dissipative:
\begin{align*}
      \norm{u(t)}_2^2
      =
      \norm{u_0}_2^2
      +2\Im\lambda\int_0^t \norm{u(s)}_{p+1}^{p+1}\,ds.
\end{align*}
The purpose of this paper is to establish the sharp large-data decay rate of
$\norm{u(t)}_2$ in the attractive-dissipative case
\[
      \Re\lambda<0,\qquad \Im\lambda<0.
\]

For general solutions, the natural range for the $L^2$-decay problem is
\[
      1<p\le 1+\frac{2}{d}.
\]
Indeed, when $p>1+2/d$, the nonlinear dissipation is too weak in general, and
the $L^2$-norm need not decay to zero; see, for instance, \cite{S21}.
Under the strong dissipative condition
\begin{align}
      (p-1)|\lambda|\le (p+1)|\Im\lambda|,
      \qquad \Im\lambda<0,
      \label{sdc}
\end{align}
Hayashi, Li, and Naumkin \cite{HLN16} proved that arbitrary data
$u_0\in\Sigma:=H^1(\R^d)\cap\mathcal{F}H^1(\R^d)$ give the decay rate
\begin{align}
      \norm{u(t)}_{2} \lesssim
      \begin{cases}
            (\log t)^{- \frac{2}{(d+2)(p-1)}} & \text{if } p=1+2/d,\\
            t^{- \frac{2}{(d+2)(p-1)} \left(1-\frac{d(p-1)}{2}\right)}
            & \text{if } 1<p<1+2/d.
      \end{cases}
      \label{hln:rate}
\end{align}

The difficulty in removing \eqref{sdc} depends on the sign of $\Re\lambda$.
When $\Re\lambda\ge0$, the standard energy gives a uniform bound for
$\norm{\nabla u(t)}_2$, and the decay rate \eqref{hln:rate} was obtained in
\cite{GKS24}.
When $\Re\lambda<0$, however, the standard energy
\[
      E(u)=\frac12\norm{\nabla u}_2^2
      +\frac{\Re\lambda}{p+1}\norm{u}_{p+1}^{p+1}
\]
is not sign definite.
Thus the standard energy does not directly control $\norm{\nabla u(t)}_2$.
Recent works \cite{GKS25} and \cite{KMS25} overcame this difficulty only in restricted
ranges of $p$; in particular, \cite{KMS25} reached
\[
      p\le 1+\frac{4}{3d}.
\]
Thus the remaining range
\[
      1+\frac{4}{3d}<p\le 1+\frac{2}{d}
\]
was left open in the attractive-dissipative case.

Dissipative nonlinear Schr\"odinger equations have also been studied in connection with refined asymptotic behavior and pointwise decay; see
\cite{CH20,CHN21,CN18,CN17,Hoshino25,Hoshino20,Hoshino19-2,Hoshino19-1}.
These works often impose smoothness, spatial decay, or additional structural conditions on the initial data.
In contrast, we prove an $L^2$-decay estimate for arbitrary data in $\Sigma$.

Another related direction concerns nonlinear Schr\"odinger equations with combined focusing and dissipative nonlinearities, where finite-time blow-up and global existence have been studied; see
\cite{ACS15,AS10,Darwich14}.
The present paper treats a single-power attractive-dissipative nonlinearity and focuses on the quantitative $L^2$-decay of global solutions in the range
$1<p\le1+2/d$.

The main observation of this paper is that the mass dissipation itself can be
used to restore a coercive energy control.
For a suitable constant $\gamma>0$, we introduce the augmented energy
\[
      E_{\mathrm{aug}}(u)
      =
      E(u)+\gamma\norm{u}_2^2.
\]
The time derivative of the added term produces the negative contribution
\[
      -2\gamma|\Im\lambda|\norm{u}_{p+1}^{p+1},
\]
which absorbs the noncoercive part of the standard energy estimate.
This bypasses the bootstrap used in earlier approaches, where intermediate
$L^2$-decay estimates are substituted into energy estimates, and yields a direct
uniform-in-time bound for $\norm{\nabla u(t)}_2$.
Consequently, we obtain \eqref{hln:rate} throughout the whole range
$1<p\le1+2/d$.

We use the standard integral formulation of solutions.
Let $I\subset\R$ be an interval and fix $t_0\in I$.
A function $u\in C(I;L^2(\R^d))$ is called a solution to \eqref{nls} on $I$ if
\[
      u(t)=U(t-t_0)u(t_0)
      -i\int_{t_0}^t U(t-s)\lambda |u(s)|^{p-1}u(s)\,ds
\]
in $L^2(\R^d)$ for all $t,t_0\in I$, where $U(t)=e^{it\Delta/2}$.

We now state the main result.

\begin{theorem} \label{thm:1}
Let $1<p \le 1+2/d$.
Assume $\Re \lambda <0$ and $\Im \lambda < 0$.
Take $u_{0} \in \Sigma$ with $\norm{u_{0}}_{\Sigma} \neq 0$.
Then \eqref{nls} admits a unique global solution $u \in C([0, \infty) ; \Sigma)$.
Moreover, if $p = 1 + 2/d$, then it holds that
\[
      \norm{ u(t)}_{2} \lesssim (\log (1+t))^{- \frac{2}{(d+2)(p-1)}}
\]
for any $t \ge 1$.
When $1<p < 1+ 2/d$, the estimate
\[
      \norm{ u(t)}_{2} \lesssim (1+t)^{- \frac{2}{(d+2)(p-1)} \left(1- \frac{d(p-1)}{2} \right)}
\]
is valid for any $t \ge 0$.
\end{theorem}

\begin{remark}
The argument also applies to the repulsive-dissipative case $\Re\lambda\ge0$, $\Im\lambda<0$.
In that case, the standard energy already gives a uniform gradient bound, and
the augmented energy argument recovers the same decay rate without imposing
\eqref{sdc}.
\end{remark}

\subsection{Strategy of the proof}
The proof is based on two elementary estimates and one augmented energy argument.
First, the $L^2$-dissipative identity and H\"older's inequality imply
\begin{align}
      \frac{d}{dt}\norm{u(t)}_2^2
      \le
      -C\norm{u(t)}_2^\beta\norm{u(t)}_q^{-(\beta-p-1)}
      \label{stra:1}
\end{align}
for some $\beta>p+1$, where $q=1$ if $d=1$ and
$q=2(d+1)/(d+2)$ if $d\ge2$.
Second, the weighted interpolation inequality
\[
      \norm{u}_q
      \lesssim
      \norm{u}_2^{1-d\alpha}
      \norm{xu}_2^{d\alpha},
      \qquad
      \alpha=\frac1q-\frac12,
\]
reduces the decay estimate to the growth control of $\norm{xu(t)}_2$.

The virial-type estimate
\[
      \norm{xu(t)}_2
      \lesssim
      \norm{xu_0}_2+\int_0^t\norm{\nabla u(s)}_2\,ds
\]
shows that it is enough to obtain a uniform-in-time bound for
$\norm{\nabla u(t)}_2$.
This is precisely where the attractive-dissipative case is difficult, because
the standard energy is not sign definite.

In previous approaches to the attractive-dissipative case, one had to improve
the growth bound for $\norm{\nabla u(t)}_2$ through an iteration argument.
Here the iteration means that a preliminary growth bound is converted through
\eqref{stra:1} into an $L^2$-decay estimate, which is then substituted into the
energy estimate to improve the growth bound.
The point of the present paper is that this iteration can be avoided.
We add a suitable multiple of the decreasing $L^2$-norm to the standard energy
and prove in Proposition \ref{prop-eneB} that
\[
      E_{\mathrm{aug}}(u(t))
      =
      E(u(t))+\gamma\norm{u(t)}_2^2
\]
is nonincreasing for an appropriate choice of $\gamma>0$ depending on the size
of the initial data.
This directly gives a uniform gradient bound, hence
\[
      \norm{xu(t)}_2\lesssim 1+t,
\]
and the decay rate \eqref{hln:rate} follows from \eqref{stra:1}.

\vskip2mm
\subsection*{Notations}
For any $p \ge 1$, $L^p = L^{p}(\R^d)$ denotes the usual Lebesgue space on $\R^d$ equipped
with the norm $\norm{\, \cdot\, }_{p} \coloneqq \norm{\, \cdot\,}_{L^{p}}$.
Set $\ev{a}=(1+|a|^2)^{1/2}$ for $a \in \C$ or $\R^d$.
$\mathcal{F}[u] = \widehat{u}$ is the usual Fourier transform of a function $u$ on $\R^d$.
For $s \ge 0$,
the standard Sobolev space on $\R^{d}$ is defined
by $H^{s} = H^{s}(\R^{d}) \coloneqq \{ u \in L^{2}(\mathbb{R}^d) \mid \norm{u}_{H^{s}} \coloneqq \norm{\ev{\nabla}^{s} u}_{L^2} < \infty \}$.
We denote the weighted Sobolev space by
$\mathcal{F}H^{s} = \mathcal{F}H^{s}(\R^{d}) \coloneqq \{ u \in L^{2}(\mathbb{R}^d)
\mid \norm{u}_{\mathcal{F}H^{s}} \coloneqq \norm{\ev{x}^{s}u}_{L^2} < \infty \}$.
Set $\Sigma = H^{1} \cap \mathcal{F} H^{1}$.
Let $U(t)$ be the Schr\"odinger group $e^{it \Delta/2}$.
$A \lesssim B$ denotes $A \le CB$ for some constant $C>0$, which may change from line to line.

\bigskip
The paper is organized as follows.
Section \ref{sec:2} collects the preliminary estimates needed for the decay argument.
Section \ref{sec:3} introduces the augmented energy method, derives the uniform gradient bound without iteration, and proves Theorem \ref{thm:1}.

\section{Preliminary} \label{sec:2}

In this section, we introduce several preliminary results that will be used in the proof of the main theorem.
Let us first recall Strichartz estimates. A pair $(p,r)$ is said to be admissible if
\[
      2\le q,r \le \infty,\quad
      \frac{2}{q} = d\left(\frac12-\frac{1}{r} \right), \quad (d,q,r)\neq (2,2,\infty).
\]

\begin{lemma}[Strichartz estimates, e.g., \cite{S77}, \cite{GV85}, \cite{Y87}, and \cite{KT98}] 
Let $(q,r)$ and $(\widetilde{q}, \widetilde{r})$ be admissible pairs.
For any interval $I \ni 0$,
\begin{align*}
      \norm{U(t) f}_{L^{q}(\R;L^{r}(\R^d))} \le{}& C_{0} \norm{ f}_{L^2(\R^d)}, \\
      \norm*{\int_0^t U(t-s) F(s)\, ds}_{L^{q}(I;L^{r}(\R^d))}
      \le{}& C_{0} \norm{ F}_{L^{\widetilde{q}'}(I;L^{\widetilde{r}'}(\R^d))},
\end{align*}
where $C_{0} >0$ is a certain constant independent of $I$, and $\rho'$ is the dual exponent of $\rho \ge 1$.
\end{lemma}

Let us define $J(t) = x+it \nabla$ for any $t \ge 0$, which is
related to the Galilean symmetry of equation \eqref{nls}.
The operator satisfies
\[
      J(t) = \mathcal{M}(t) it \nabla \mathcal{M}(-t) = U(t) x U(-t),
       \quad \mathcal{M}(t) \coloneqq e^{i|x|^2/2t}.
\]

We state the existence of a unique global $\Sigma$-solution to \eqref{nls} and its $L^2$-dissipative property. The proof is based on the standard arguments involving Strichartz estimates and the sign condition $\Im \lambda <0$.

\begin{lemma}[Global well-posedness and $L^2$-dissipation] \label{lem:22}
Let $1<p < 1+4/d$ and $\Im \lambda < 0$. Suppose $ u_{0} \in \Sigma$. Then \eqref{nls} admits a unique global solution $u \in C([0, \infty); \Sigma)$ satisfying
\begin{align}
      \norm{u(t)}_{2} \le{}& \norm{u_{0}}_{2} \label{uni:1}
\end{align}
for any $t \ge 0$.
\end{lemma}
\begin{proof}[Proof of Lemma \ref{lem:22}]
The existence part is standard, but we sketch the proof for reader's convenience.
Let us introduce the complete metric space
\begin{align*}
      X_{T} \coloneqq{}& \left\{ u \in C([0, T] ; L^{2}) \cap L^{q_{0}}(0, T; L^{r_{0}})
\mid \norm{ u}_{X_{T}} \le 2 C_{0} \norm{u_{0}}_{2} \right\}, \\
      \norm{ u}_{X_{T}} \coloneqq{}& \norm{ u}_{L^{ \infty}(0, T; L^{2})} + \norm{ u}_{L^{q_{0}}(0, T; L^{r_{0}})},
      \quad d( u, v) \coloneqq \norm{u-v}_{X_{T}},
\end{align*}
where $r_{0} = p+1$, $2/q_{0} = d/2 - d/r_{0}$ and $C_{0}>0$ is the certain constant arising from the Strichartz estimate.
In what follows, we simply write $L^{q}(0, T; L^{r})$ as $L^{q}_{T}L^{r}$.
The standard well-posedness theory implies that there exist $T>0$ and a unique solution $u \in X_{T}$
to \eqref{nls} with $u \in L^{q}_{T} L^{r}$ for any admissible pair $(q,r)$ (for details, see \cite{YT87}).
Thanks to $\Im \lambda \le 0$,
multiplying \eqref{nls} by $\overline{ u}$, taking the imaginary part of both sides, and integrating over $\R^{d}$, we obtain
\begin{align*}
      \frac{1}{2} \frac{d}{dt} \norm{u(t)}_{2}^{2} =  \Im \lambda \norm{u(t)}_{p+1}^{p+1}
      \le 0,
\end{align*}
which implies $\norm{u(t)}_{2} \le \norm{u_{0}}_{2}$ for any $t \in [0, T].$
Then the $L^{2}$-solution $u$ can be extended globally in time.
Further, since $p<1+4/d$ and $u_{0} \in \Sigma$, using $J(t) = \mathcal{M}(t) it \nabla \mathcal{M}(-t)$,
we obtain $\nabla u$, $J(\cdot) u \in C([0, \infty) ; L^{2}) \cap L^{q}_{loc}((0, \infty); L^{r})$
from the persistence of regularity of solutions via a standard bootstrap argument involving Strichartz estimate.
\end{proof}

\vskip2mm
To evaluate the time decay rate, we introduce a general virial-type estimate for the weighted norm $\norm{x u(t)}_2$, provided that the gradient norm $\norm{\nabla u(t)}_2$ is uniformly bounded.

\begin{lemma}[Virial-type estimate] \label{lem:uni}
Assume $\Re \lambda < 0$ and $\Im \lambda < 0$. Let $u$ be the global solution given by Lemma \ref{lem:22}. Suppose that there exists a constant $M > 0$ such that
\[
      \sup_{t \ge 0} \norm{\nabla u(t)}_{2} \le M.
\]
Then the solution satisfies $u \in C([0, \infty) ; \mathcal{F}H^{1})$.
Moreover, the estimate
\[
      \norm{x u(t)}_{2} \lesssim 1+ t
\]
holds for any $t \ge 0$.
\end{lemma}
\begin{proof}
Noting $J(t) = x + it \nabla$, it follows that
\begin{align*}
      \norm{x (u(t) - u(t_{0}) )}_{2}
      \le{}& \norm{J(t) u(t) - J(t_{0}) u(t_{0})}_{2} + |t-t_{0}| \norm{\nabla u(t_{0})}_{2} \\
      &+ |t| \norm{\nabla (u(t) - u(t_{0}))}_{2} \\
      \rightarrow{}& 0
\end{align*}
as $t \rightarrow t_{0}$, which implies $u \in C([0, \infty) ; \mathcal{F}H^{1})$.
Next, we consider the case $\Re \lambda < 0$ since the other case is handled similarly.
Multiplying \eqref{nls} by $|x|^{2} \overline{u}$, taking the imaginary part of both sides, and integrating over $\R^{d}$, we have
\begin{align}
      \frac{1}{2} \frac{d}{dt} \norm{ x u(t)}_{2}^{2} \le \Im \int_{\R^{d}} \overline{u} x \cdot \nabla u\, dx.
      \label{eq:2}
\end{align}
Using \eqref{eq:2} and the assumption $\norm{\nabla u(t)}_{2} \le M$, one deduces from H\"older's inequality that
\begin{align*}
      \frac{d}{dt} \norm{x u(t)}_{2} \le \norm{ \nabla u(t)}_{2}
      \le M,
\end{align*}
which implies the desired estimate. Thus the proof is completed.
\end{proof}


The following weighted interpolation inequality can be regarded as a dual version of Gagliardo-Nirenberg inequality.
The first author used it in \cite{K20a} to obtain an $L^2$-decay estimate for \eqref{nls}, and it will play a crucial role to prove our main result:
\begin{lemma}[Weighted interpolation inequality] \label{lem:in1}
Fix $m \in \N$.
Let $q \in [1, 2)$ if $d =1$ and $q \in (2d/(d+2m), 2)$ if $d \ge 2$.
Set $\alpha = 1/q - 1/2$.
Then it holds that
\[
      \norm{ f}_{q} \lesssim \norm{ f}_{2}^{1- \frac{d \alpha}{m}} \norm{|x|^{m} f}_{2}^{ \frac{d \alpha}{m}}.
\]
\end{lemma}
\begin{proof}
The case $f \equiv 0$ is trivial; otherwise, by scaling and H\"older's inequality, we have
\begin{align*}
      \norm{ f}_{q}
      \le{}& \lambda^{\frac{d}{q}} \norm{(1+ |y|^{m})^{-1}}_{2q/(2-q)} \norm{ (1+|y|^{m}) f( \lambda \cdot)}_{2} \\
      \lesssim{}& \lambda^{d \alpha}  \left( \norm{ f}_{2} + \lambda^{-m} \norm{|x|^{m} f}_{2}  \right),
\end{align*}
from which the assertion follows, taking $\lambda = (\norm{|x|^{m} f}_{2}/ \norm{ f}_{2})^{1/m}$.
The proof is completed.
\end{proof}

\section{Uniform Gradient Bound and Proof of the Main Result} \label{sec:3}

We prove the uniform gradient bound by using an augmented energy.
Set
\begin{align*}
      E(u) \coloneqq \frac{1}{2} \norm{\nabla u}_{2}^{2}
      + \frac{\Re\lambda}{p+1}\norm{u}_{p+1}^{p+1},
\end{align*}
and, for $\gamma>0$,
\begin{align*}
      E_{\mathrm{aug}}(u) \coloneqq E(u)+\gamma\norm{u}_2^2.
\end{align*}
The added term produces the dissipative contribution
$-2\gamma|\Im\lambda|\norm{u}_{p+1}^{p+1}$ in the time derivative.

\begin{proposition}[Uniform gradient bound]\label{prop-eneB}
Let $1< p < 1+4/d$, $\Re \lambda < 0$, and $\Im \lambda < 0$. Let $u \in C([0,\infty); \Sigma)$ be the global solution to \eqref{nls} given by Lemma \ref{lem:22}. Then there exists a constant $\gamma = \gamma(\norm{u_0}_{H^1})>0$ such that, for any $t \ge 0$,
\begin{align}\label{ene-B}
      E_{\mathrm{aug}}(u(t)) \le E_{\mathrm{aug}}(u_0).
\end{align}
Consequently, there exists a constant $M(u_{0}) > 0$ independent of $t$ such that
\begin{align}
      \norm{ \nabla u(t)}_{2} \le M(u_{0}) \label{uni:3a}
\end{align}
for any $t \ge 0$.
\end{proposition}

\begin{proof}[Proof of Proposition \ref{prop-eneB}]
Let $\lambda_1 = \Re \lambda$ and $\lambda_2 = \Im \lambda$.
Multiplying \eqref{nls} by $\overline{ \partial_{t} u}$, taking the real part of both sides,
and integrating over $\mathbb{R}^{d}$, we have
\begin{align*}
      \frac{d}{dt} E(u(t)) = 2 \lambda_{2} \int_{\mathbb{R}^{d}} |u(t)|^{p+1} \Im (u(t) \overline{\partial_{t} u(t)}) \, dx.
\end{align*}
Using \eqref{nls} again, one estimates
\begin{align}
\begin{aligned}
      \frac{d}{dt} E(u(t)) ={}& \frac{\lambda_{2}(p+1)}{2} \int_{\R^{d}} |u|^{p-1} | \nabla u|^{2}\, dx \\
&+ \frac{\lambda_{2}(p-1)}{2} \int_{\mathbb{R}^{d}} |u|^{p-3} \Re \big(u^{2} \overline{ \nabla u}^{2}\big)\, dx \\
      &{}+ 2\lambda_{1}\lambda_{2} \|u\|_{2p}^{2p} \\
      \le{}& \lambda_{2} \int_{\mathbb{R}^{d}} |u|^{p-1} |\nabla u|^{2}\, dx + 2\lambda_{1}\lambda_{2} \|u\|_{2p}^{2p}.
\end{aligned}
      \label{en:ineq3}
\end{align}
Recall that we assume $\lambda_1 < 0$.
Set $f = |u|^{\frac{p+1}{2}}$. Note that $\big|\nabla |u|\big| \le |\nabla u|$, which implies $\int |u|^{p-1} |\nabla u|^2 dx \ge \frac{4}{(p+1)^2} \norm{\nabla f}_2^2$.
By Gagliardo-Nirenberg inequality:
\[
      \|f\|_{\frac{4p}{p+1}}
      \le{}C\|f\|_{2}^{1- \theta}\|\nabla f\|_{2}^{\theta},
      \quad \theta = \frac{d(p-1)}{4p} \in (0,1),
\]
we see from \eqref{en:ineq3} that
\begin{align*} 
      \frac{d}{dt} E(u(t)) \le{}& - \frac{4|\lambda_{2}|}{(p+1)^2}  \|\nabla f\|_{2}^{2}
      + 2\lambda_{1}\lambda_{2}  \|u\|_{2p}^{2p} \\
      \le{}& - \frac{4|\lambda_{2}|}{(p+1)^2} \|\nabla f\|_{2}^{2}
      + 2\lambda_{1}\lambda_{2} C^{\frac{4p}{p+1}}
      \|f\|_{2}^{\frac{4p-d(p-1)}{p+1}}\|\nabla f\|_{2}^{\frac{d(p-1)}{p+1}}.
\end{align*}
Since $d(p-1)/(p+1) < 2$ under the assumption $p < 1+4/d$, applying Young's inequality $ab \le C_\varepsilon a^{\frac{2(p+1)}{2(p+1)-d(p-1)}}+\varepsilon b^{\frac{2(p+1)}{d(p-1)}}$ with a sufficiently small $\varepsilon>0$ to absorb the $\|\nabla f\|_{2}^{2}$ term, we obtain
\begin{align}\label{d-dt-E}
\begin{aligned}
      \frac{d}{dt} E(u(t))
      \le{}&C_E |\lambda_2|\|f\|_{2}^{\frac{4p-d(p-1)}{p+1}\cdot\frac{2(p+1)}{2(p+1)-d(p-1)}}\\
      ={}&C_E |\lambda_2|\big(\|u\|_{p+1}^{p+1}\big)^{\frac{4p-d(p-1)}{2(p+1)-d(p-1)}},
\end{aligned}
\end{align}
where $C_E > 0$ is a constant arising from the previous estimate,
and $p_1 \coloneqq \frac{4p-d(p-1)}{2(p+1)-d(p-1)}>1$ since $p>1$.
Combining \eqref{d-dt-E} and the $L^2$-dissipative identity
$\frac{d}{dt}\|u(t)\|_{2}^2=2\lambda_2\|u(t)\|_{p+1}^{p+1}$,
we have
\begin{align}\label{Ene-L2}
\frac{d}{dt} \big(E(u(t)) +\|u(t)\|_2^2\,\big) \le g(\|u(t)\|_{p+1}^{p+1}),
\end{align}
where
\begin{align*}
      g(X)={}&C_E |\lambda_2| X^{p_1} -2|\lambda_2|X.
\end{align*}
Since $p_1>1$, $g(X)$ has a positive root $\alpha$ such that $g(\alpha)=0$, and
it follows that $g(X) \le 0$ for all $X \in [0, \alpha]$.
We first claim that there exists a sufficiently small $\varepsilon_0>0$ such that
if $\|u_0\|_{H^1}^{p+1} \le \varepsilon_0$, then
\begin{align}\label{cl:2p}
\sup_{t \ge 0}\|u(t)\|_{p+1}^{p+1} < \alpha.
\end{align}
Note that the embedding $H^{1} \hookrightarrow L^{p+1}$ and $u \in C([0, \infty); H^{1})$ imply
\begin{align*}
      u \in C([0, \infty); L^{p+1}).
\end{align*}
To show \eqref{cl:2p}, set
\begin{align*}
T_\alpha(u_0)=\inf{\{T>0 \mid \|u(T)\|_{p+1}^{p+1}=\alpha \}}
\end{align*}
for solutions $u$ to \eqref{nls}.
For sufficiently small initial data $\|u_0\|_{H^1}$, we find $T_\alpha(u_0)>0$
from the time continuity of $\|u(t)\|_{H^1}$.
We shall prove \eqref{cl:2p} by contradiction.
Assume that $T_\alpha(u_0)<\infty$ for certain initial data with
$\|u_0\|_{H^1}^{p+1} \le \varepsilon_0$, where $\varepsilon_0 > 0$ will be chosen later to be sufficiently small.
By the definition of $T_\alpha(u_0)$, we have $\|u(t)\|_{p+1}^{p+1} < \alpha$ for all $t \in [0,T_\alpha(u_0))$.
We then deduce from \eqref{Ene-L2} that
\begin{align}\label{small-EL2}
E(u(t))+\|u(t)\|_2^2 \le E(u_0)+\|u_0\|_2^2
\end{align}
for any $t \in [0,T_\alpha(u_0)]$. Using the inequalities
\[
  \|u\|_{p+1}
\le C \|u\|_{L^2}^{1-\frac{d(p-1)}{2(p+1)}}\|\nabla u\|_{L^2}^{\frac{d(p-1)}{2(p+1)}}
\]
and  $XY \le X^{\frac{q}{q-1}}+Y^q$ with $q=4/(d(p-1))$ (note that $p < 1+4/d$ ensures $q>1$),
we have
\begin{align*}
      \|\nabla u(t)\|_{2}^2
      \le CE(u_0) +C\|u_0\|_{2}^{\frac{2((2-d)p+2+d)}{4-d(p-1)}} + C\|u_0\|_2^2
\end{align*}
for any $t \in [0,T_\alpha(u_0)]$, where $C>0$ does not depend on $t$ and $u_0$.
From the above energy estimate and Gagliardo-Nirenberg inequality,
we derive that
$\|u(T_\alpha(u_0))\|_{p+1}^{p+1} \le C_{1} \varepsilon_0$, where the constant $C_{1}>0$ does not depend on $t$ and $u_0$.
Taking $\varepsilon_{0} >0$ such that $C_{1} \varepsilon_{0} < \alpha$, this contradicts the definition of $T_\alpha(u_0)$ because it forces $\|u(T_\alpha(u_0))\|_{p+1}^{p+1} < \alpha$. Thus $T_\alpha(u_0)=\infty$ holds.
This proves \eqref{cl:2p}.
Thanks to \eqref{Ene-L2} and \eqref{cl:2p}, we obtain \eqref{small-EL2} for any $t \ge 0$, and hence \eqref{ene-B} and \eqref{uni:3a} are valid under the smallness assumption $\|u_0\|_{H^1} \ll 1$.

To obtain \eqref{ene-B} and \eqref{uni:3a} without any size restriction,
we apply the scaling transformation $u_\mu(t,x)=\mu^{2/(p-1)}u(\mu^2t,\mu x)$ to the solutions.
Then we see that
\[
      \|u_{\mu}(0)\|_{H^1}
      \lesssim
      \mu^{\frac{2}{p-1}-\frac{d}{2}}\|u_0\|_{2}
      +\mu^{\frac{2}{p-1}+1-\frac{d}{2}}\|\nabla u_0\|_{2}
      \to 0
\]
as $\mu \to 0$ since $p < 1+4/d$.
Therefore, by choosing $\mu = \mu(\norm{u_0}_{H^1}) > 0$ small enough, we can ensure $\norm{u_{\mu}(0)}_{H^1} \ll 1$. Indeed, let $C_H>0$ be a constant such that, for $0<\mu\le1$,
\[
      \|u_{\mu}(0)\|_{H^1}
      \le C_H\mu^{\frac{2}{p-1}-\frac{d}{2}}\|u_0\|_{H^1}.
\]
Setting
\[
      \mu = \min \left( \left( \frac{\varepsilon_0^{\frac{1}{p+1}}}{C_H\norm{u_0}_{H^1}} \right)^{\left(\frac{2}{p-1} - \frac{d}{2}\right)^{-1}}, 1 \right),
\]
one obtains $\norm{u_\mu(0)}_{H^1} \le \varepsilon_0^{\frac{1}{p+1}}$.
Applying the previous small data result to $u_\mu$, we obtain
\begin{align*}
E(u_\mu(t))+\norm{u_\mu(t)}_2^2 \le E(u_{\mu}(0))+\norm{u_{\mu}(0)}_2^2
\end{align*}
for any $t \ge 0$, even if $\norm{u_0}_{H^1} \gg 1$.
Scaling back to the original solution $u(t,x)$, this inequality is equivalent to
\begin{align*}
      \mu^{\frac{4}{p-1}-d+2} E(u(\mu^2 t)) + \mu^{\frac{4}{p-1}-d} \|u(\mu^2 t)\|_2^2
      \le \mu^{\frac{4}{p-1}-d+2} E(u_0) + \mu^{\frac{4}{p-1}-d} \|u_0\|_2^2.
\end{align*}
Dividing both sides by $\mu^{\frac{4}{p-1}-d+2}$ and replacing $\mu^2 t$ with $t \ge 0$, we deduce
\begin{align*}
      E(u(t)) + \mu^{-2} \|u(t)\|_2^2 \le E(u_0) + \mu^{-2} \|u_0\|_2^2.
\end{align*}
This directly yields \eqref{ene-B} by setting $\gamma \coloneqq \mu^{-2}$, which explicitly depends on the initial size $\|u_0\|_{H^1}$.
The uniform gradient bound \eqref{uni:3a} then follows from \eqref{ene-B} and Gagliardo-Nirenberg inequality.
Indeed, since $\Re\lambda<0$, we have
\[
      \frac12 \norm{\nabla u(t)}_2^2
      = E(u(t)) + \frac{|\Re\lambda|}{p+1}\norm{u(t)}_{p+1}^{p+1}.
\]
By Gagliardo-Nirenberg inequality and Young's inequality,
\[
      \norm{u(t)}_{p+1}^{p+1}
      \le \varepsilon \norm{\nabla u(t)}_2^2
      + C_\varepsilon
      \norm{u(t)}_2^{\frac{2((2-d)p+2+d)}{4-d(p-1)}}
\]
for sufficiently small $\varepsilon>0$.
Using \eqref{ene-B} and the monotonicity of $\norm{u(t)}_2$, we obtain
\[
      \norm{\nabla u(t)}_2^2
      \le
      C\left(
      |E_{\mathrm{aug}}(u_0)|
      + \norm{u_0}_2^{2}
      + \norm{u_0}_2^{\frac{2((2-d)p+2+d)}{4-d(p-1)}}
      \right)
\]
for all $t\ge0$, where $C>0$ depends on $d,p,\lambda$ and $\gamma$.
This gives \eqref{uni:3a}.
\end{proof}

We now prove the main theorem by combining the uniform gradient bound with the virial-type and weighted interpolation inequalities.

\begin{proof}[Proof of Theorem \ref{thm:1}]
We follow the argument in \cite{K20a} and \cite{GKS24}.
Let $u \in C([0, \infty); \Sigma)$ be the global solution given by Lemma \ref{lem:22}.
By Proposition \ref{prop-eneB}, the gradient norm is uniformly bounded, which allows us to apply Lemma \ref{lem:uni}.
Multiplying \eqref{nls} by $\overline{ u}$, taking the imaginary part of both sides and integrating over $\R^{d}$,
we see from $\Im \lambda < 0$ and H\"older's inequality that
\begin{align*}
      \frac{1}{2} \frac{d}{dt} \norm{u(t)}_{2}^{2}
      = \Im \lambda \norm{ u(t)}_{p+1}^{p+1}
      \le -|\Im \lambda| \norm{u(t)}_{2}^{\beta} \norm{u(t)}_{q}^{p+1 - \beta},
\end{align*}
where $q =1$ if $d=1$ and $q = 2(d+1)/(d+2)$ if $d \ge 2$, and $\beta>p+1$ is defined by
\[
      \beta \alpha = \frac{p+1}{q} -1, \quad
      \alpha = \frac{1}{q} - \frac{1}{2}.
\]
Noting $p+1-\beta<0$ and $(p+1 - \beta)d \alpha = d(1-p)/2$,
Lemma \ref{lem:in1} with $m=1$ leads to
\begin{align*}
      \norm{u(t)}_{q}^{p+1 - \beta}
      &\gtrsim \norm{u(t)}_{2}^{(p+1 - \beta)(1- d \alpha)} \norm{x u(t)}_{2}^{(p+1 - \beta) d \alpha} \\
      &= \norm{u(t)}_{2}^{p+1 + \frac{d(p-1)}{2}- \beta} \norm{x u(t)}_{2}^{\frac{d(1-p)}{2}}.
\end{align*}
Combining the above with Lemma \ref{lem:uni}, we obtain
\begin{align*}
      \frac{1}{2} \frac{d}{dt} \norm{u(t)}_{2}^{2}
      \le{}& - C \norm{u(t)}_{2}^{p+1+ \frac{d(p-1)}{2}} \norm{x u(t)}_{2}^{\frac{d(1-p)}{2}} \notag\\
      \le{}& - C (1+t)^{-\frac{d(p-1)}{2}} \norm{u(t)}_{2}^{p+1+ \frac{d(p-1)}{2}},
\end{align*}
which yields
\begin{align*}
      \frac{d}{dt}  \left( \norm{u(t)}_{2}^{-\frac{(d+2)(p-1)}{2}}  \right) \ge C (1+t)^{-\frac{d(p-1)}{2}}
\end{align*}
for any $t \ge 0$.
Noting \eqref{uni:1} and $p \le 1+2/d$, integrating the above with respect to $t$, we establish the desired decay estimate.
The proof is completed.
\end{proof}

\subsection*{Acknowledgments}
N.K. was supported by JSPS KAKENHI Grant Numbers 23K03168 and 23K03183.
H.M. was supported by JSPS KAKENHI Grant Number 22K13941 and 26K00612.
T.S. was supported by JSPS KAKENHI Grant Numbers 22K13937 and 23KJ1765.


\end{document}